\documentclass[12pt]{article} 
\usepackage{graphicx}
\usepackage{amsmath,amssymb,textcomp,graphicx,amsthm} 
\usepackage{epsfig}
\usepackage{color}
\usepackage[alphabetic]{amsrefs}
 \usepackage{verbatim}

\numberwithin{equation}{section}

\newcommand{\e}{\varepsilon}

\newcommand{\R}{\mathbb{R}}

\newtheorem{de}{Definition}
\newtheorem{lem}{Lemma}
\newtheorem{prop}{Proposition}
\newtheorem{thm}{Theorem}

\newtheorem{rem}{Remark}

\newcommand{\lap}{\Delta}

\newcommand{\pv}{\operatorname{PV}}

\renewcommand{\div}{\operatorname{div}}
\newcommand{\osc}{\operatorname{osc}}

\title{H\"older estimates for viscosity solutions of equations of fractional $p$-Laplace type}

\author{Erik Lindgren\footnote{Department of Mathematics, KTH. Supported by the Swedish Research Council, grant no. 2012-3124. Partially supported by the Royal Swedish Academy of Sciences. A substantial part of this work was carried out at the Isaac Newton's Institute, during the program ``Free boundary problems'' and I am thankful for the great hospitality.}}
\begin{document}
\maketitle
\begin{abstract}\noindent We prove H\"older estimates for viscosity solutions of a class of possibly degenerate and singular equations modelled by the fractional $p$-Laplace equation
$$
\pv \int_{\R^n}\frac{|u(x)-u(x+y)|^{p-2}(u(x)-u(x+y))}{|y|^{n+sp}}\, dy =0,
$$
where $s\in (0,1)$ and $p>2$ or $1/(1-s)<p<2$. Our results also apply for inhomogeneous equations with more general kernels, when $p$ and $s$ are allowed to vary with $x$, without any regularity assumption on $p$ and $s$. This complements and extends some of the recently obtained H\"older estimates for weak solutions.
\end{abstract}


\section{Introduction} 
We study the local H\"older regularity for viscosity solutions of possibly degenerate and singular non-local equations of the form
\begin{equation*}
\pv \int_{\R^n}|u(x)-u(x+y)|^{p-2}(u(x)-u(x+y))K(x,y)\, dy = f(x),
\end{equation*}
where $f$ is bounded and $K(x,y)$ essentially behaves like $|y|^{-n-sp}$. Here $\pv$ stands for the \emph{principal value}.

This type of equations is one possible non-local counterpart of equations of $p$-Laplace type and arises for instance as the Euler-Lagrange equation of functionals in fractional Sobolev spaces. Solutions can also be constructed directly via Perron's method (cf. \cite{IN10}). In the case $K(y)=|y|^{-n-sp}$, when properly rescaled, solutions converge to solutions of the $p$-Laplace equation
$$
\lap_p u=\div (|\nabla u|^{p-2}\nabla u)=0
$$ as the parameter $s$ tends to 1, see \cite{IN10}.

Our first and main result is that bounded viscosity solutions (see Section \ref{sec:visc}) of the homogeneous equation are locally H\"older continuous, see the theorem below. Throughout the paper we denote by $B_r$, the ball of radius $r$ centered at the origin.

\begin{thm}\label{thm:main} Assume $K$ satisfies $K(x,y)=K(x,-y)$ and there exist $\Lambda\geq \lambda>0$, $M>0$ and $\gamma>0$ such that
\begin{align*}
\frac{\lambda}{|y|^{n+sp}}\leq &K(x,y)\leq \frac{\Lambda}{|y|^{n+sp}}, \text{ for } y\in B_2,x\in B_2, \\
0\leq &K(x,y) \leq \frac{M}{|y|^{n+\gamma}}, \text{ for } y\in \R^n\setminus B_\frac14, x\in B_2, 
\end{align*}
where $s\in (0,1)$ and $p\in (1,\infty)$. In the case $p<2$ we require additionally $p>1/(1-s)$. Let $u\in L^\infty(\R^n)$ be a viscosity solution of 
$$
Lu:=\pv \int_{\R^n}|u(x)-u(x+y)|^{p-2}(u(x)-u(x+y))K(x,y)\, dy =0\text{ in }B_2.
$$
Then $u$ is H\"older continuous in $B_1$ and in particular there exist $\alpha$ and $C$ depending on $\lambda,\Lambda,M,p,s$ and $\gamma$ such that 
$$
\|u\|_{C^\alpha(B_1)}\leq C\|u\|_{L^\infty(\R^n)}.
$$
\end{thm}
In particular, Theorem \ref{thm:main} applies for the fractional $p$-Laplace equation
$$
\pv \int_{\R^n}\frac{|u(x)-u(x+y)|^{p-2}(u(x)-u(x+y))}{|y|^{n+sp}}\, d y=0.
$$
We are also able to prove H\"older estimates for inhomogeneous equations with variable exponents, see below:

\begin{thm}\label{thm:main2} Assume $K$ satisfies $K(x,y)=K(x,-y)$ and there exist $\Lambda\geq \lambda>0$, $M>0$ and $\gamma>0$ such that
\begin{align*}
\frac{\lambda}{|y|^{n+s(x)p(x)}}\leq &K(x,y)\leq \frac{\Lambda}{|y|^{n+s(x)p(x)}}, \text{ for } y\in B_2,x\in B_2,\\
0\leq &K(x,y) \leq \frac{M}{|y|^{n+\gamma}}, \text{ for } y\in \R^n\setminus B_\frac14, x\in B_2, 
\end{align*}
where $0<s_0<s(x)<s_1<1$ and $1<p_0<p(x)<p_1<\infty$. In the case $p(x)<2$ we require additionally  that there is $\tau>0$ such that
$$
p(x)(1-s(x))-1>\tau.$$
Let $f\in C(B_2)\cap L^\infty(B_2)$ and let $u\in L^\infty(\R^n)$ be a viscosity solution of
$$
Lu:=\pv \int_{\R^n}|u(x)-u(x+y)|^{p(x)-2}(u(x)-u(x+y))K(x,y)\, dy = f(x)\text{ in } B_2.$$
Then $u$ is H\"older continuous in $B_1$ and in particular there exist $\alpha$ and $C$ depending on $\lambda,\Lambda,M,p_0,p_1,s_0,s_1,\gamma $ and $\tau$ such that 
$$
\|u\|_{C^\alpha(B_1)}\leq C\left(\|u\|_{L^\infty(\R^n)}+\max\left(\|f\|_{L^\infty(B_2)}^\frac{1}{p_0-1},\|f\|_{L^\infty(B_2)}^\frac{1}{p_1-1}\right)\right).
$$
\end{thm}

\begin{rem} It might seem odd that the two conditions on $K$ in our main theorems are supposed to be satisfied in overlapping regions, $B_2$ and $B_\frac14$. This is only for notational convenience. It would be sufficient to have the first condition satisfied in $B_\rho$ for some $\rho>0$ and the second one satisfied outside $B_R$ for some large $R$ as long as we ask $K$ to be bounded in $B_R\setminus B_\rho$. 
\end{rem}

\subsection{Known results}
Equations similar to the ones in Theorem \ref{thm:main} were, to the author's knowledge, introduced in \cite{IN10}, where existence and uniqueness is established. It is also shown that the solutions converge to solutions of the $p$-Laplace equation, as $s\to 1$. Similar equations were also studied in \cite{CLM12}, where the focus lies in the asymptotic behaviour as $p\to \infty$. Related equations have also been suggested to be used in image processing and machine learning, see \cite{EDL12} and \cite{EDLL14}.

Recently, in \cite{CKP13} and \cite{CKP14}, H\"older estimates and a Harnack inequality were obtained for weak solutions of a very general class of equations of this type. The difference between these results and the ones in the present paper can be seen as the difference between equations in divergence form and those in non-divergence form in the non-local setting. In other words, their results are more in the flavour of Di Giorgi-Nash-Moser (cf. \cite{DG57}, \cite{Nas58} and \cite{Mos61}) while the results in this paper are more in the flavour of Krylov-Safonov (cf. \cite{KS79}).

In the case $p=2$, corresponding to equations of the form
\begin{equation}\label{eq:fraclap}
\pv \int_{\R^n}  \frac{u(x)-u(x+y)}{|y|^{n+2s}}\, dy =f(x),
\end{equation}
a similar development has already taken place. In \cite{Sil06}, a surprisingly simple proof of H\"older estimates for viscosity solutions were given for a very general class of equations corresponding to equations of non-divergence form. An adaptation of the method used therein is used in the present paper. In \cite{Kas09}, H\"older estimates were obtained for weak solutions for a class of equations corresponding to equations of divergence form, including equations of the form \eqref{eq:fraclap}.

Related is also \cite{BCF12a} and \cite{BCF12b}, where another type of degenerate (or singular) non-local equation is studied. H\"older estimates and some higher regularity theory are established. It is also proved that these equations approach the $p$-Laplace equation in the local limit.
\subsection{Comments on the equation}
Let us very briefly point out the difference between the class of equations considered in \cite{CKP13} and \cite{CKP14}, and the class of equations considered here (see also \cite{Sil06} for a similar discussion). There, weak solutions are considered, in the sense that
\begin{equation}\label{eq:weak}
\int_{\R^n}\int_{\R^n}|u(x)-u(y)|^{p-2}(u(x)-u(y))(\phi(x)-\phi(y))G(x,y)\,dx dy=0
\end{equation}
for any $\phi\in C_0^\infty(B_2)$, where $G(x,y)$ behaves like $|x-y|^{-n-sp}$. These solutions arise for instance as minimizers of functionals of the form
$$
\int_{\R^n}\int_{\R^n}|u(x)-u(y)|^{p}G(x,y)\,dx dy.
$$
In the most favorable of situations, we are allowed to change the order of integration and write \eqref{eq:weak} as
$$
 \int_{\R^n}\int_{\R^n}|u(x)-u(y)|^{p-2}(u(x)-u(y))(G(x,y)+G(y,x))\phi(x)\,dx dy=0,
$$
and conclude
$$
\pv \int_{\R^n}|u(x)-u(y)|^{p-2}(u(x)-u(y))(G(x,y)+G(y,x))\,dy=0.
$$
The change of variables $y=z+x$ yields
$$
\pv \int_{\R^n}|u(x)-u(z+x)|^{p-2}(u(x)-u(z+x))(G(x,z+x)+G(z+x,x))\,dz=0,
$$
or
$$
\pv \int_{\R^n}|u(x)-u(z+x)|^{p-2}(u(x)-u(z+x))K(x,z)\,dz=0,
$$
where $K(x,z)=G(x,z+x)+G(z+x,x)$. Then necessarily $K(x,z-x)=K(z,x-z)$. Moreover, we are not always allowed to perform the transformations above. Hence, the two types of equations overlap but neither is contained in the other. In other words, the results in \cite{CKP13} and \cite{CKP14} do not always apply to the equations considered in this paper, and vice versa, the results in this paper do not always apply to the equations studied therein.

Another important remark is that the estimates obtained in this paper are not uniform as $s\to 1$, i.e., in the limit in which the equation becomes local. This is also the case in \cite{Sil06}. For fully nonlinear equations of fractional Laplace type, uniform estimates as $s\to 1$ have been obtained (see for instance \cite{CS09}), but they are more involved, and they follow the same strategy as the estimates for fully nonlinear (local) equations.

In our case, if $\phi \in C^2_0$ and $p>2$, then 
$$
(1-s)\pv \int_{\R^n}\frac{|\phi(x)-\phi(x+y)|^{p-2}(\phi(x)-\phi(x+y))}{|y|^{n+sp}}\,d y \to -C_{p,n}\lap_p \phi,
$$
as $s\to 1$. If we instead have a kernel of the form 
$$G\left(\frac{y}{|y|}\right)\frac{1}{|y|^{n+sp}},$$
then
\begin{align*}
&(1-s)\pv \int_{\R^n}\frac{|\phi(x)-\phi(x+y)|^{p-2}(\phi(x)-\phi(x+y))G(\frac{y}{|y|})}{|y|^{n+sp}}\,d y \\
&\to -C_{p,n}|\nabla \phi|^{p-2}a_{ij}(\nabla \phi)D_{ij}^2 \phi, 
\end{align*}
as $s\to 1$, where the matrix $(a_{ij})(\nabla \phi)$ is positive definite and can be given explicitly as integrals over the sphere in terms of $G$. This type of degenerate (or singular) equations of non-divergence form, remained fairly unstudied until quite recently. Starting with \cite{BID04}, these equations have attracted an increasing amount of attention. See also \cite{Imb11} and \cite{IS13} where $C^\alpha$ and $C^{1,\alpha}$-estimates are established, respectively.

\section{Viscosity solutions}\label{sec:visc}
In this section, we introduce the notion of viscosity solutions (as in \cite{CS09}) and prove that viscosity solutions can be treated almost as classical solutions.

\begin{de} Let $D$ be an open set and let $L$ be as defined in Theorem \ref{thm:main} or Theorem \ref{thm:main2}. A function $u\in L^\infty(\R^n)$ which is upper semicontinuous in ${D}$ is a subsolution of 
$$  
L u\,\leq C \text{ in $D$}
$$
if the following holds: whenever $x_0\in D$ and $\phi\in C^2({B_r(x_0)})$ for some $r>0$ are such that
$$
\phi(x_0)=u(x_0), \quad \phi(x)\geq u(x) \text{ for $x\in B_r(x_0)\subset D$}
$$
then we have
$$
L\phi_r\, (x_0)\leq 0,
$$
where
$$
\phi_r =\left\{\begin{array}{lr}\phi \text{ in }B_r(x_0),\\
u\text{ in }\R^n\setminus B_r(x_0).
\end{array}\right. 
$$
A supersolution is defined similarly and a solution is a function which is both a sub- and a supersolution.
\end{de}
The following result verifies that whenever we can touch a subsolution from above with a $C^2$ function, we can treat the subsolution as classical subsolution. The proof is almost identical to the one of Theorem 2.2 in \cite{CS09}.
\begin{prop}\label{prop:pw} Assume the hypotheses of Theorem \ref{thm:main} or Theorem \ref{thm:main2}. Suppose $Lu\leq C$ in $B_1$ in the viscosity sense and that $x_0\in B_1$ and $\phi\in C^2({B}_r(x_0))$ is such that
$$
\phi(x_0)=u(x_0),\quad \phi(x)\geq u(x)\text{ in }B_r(x_0)\subset B_1, 
$$
for some $r>0$. Then $Lu$ is defined pointwise at $x_0$ and $Lu\, (x_0)\leq C$.  
\end{prop}
\begin{proof} Since the result is only concerned with the behavior at one fixed point $x_0$, we see that there is no difference between assuming the hypotheses of Theorem \ref{thm:main} or Theorem \ref{thm:main2}. Hence, we give the proof under the hypotheses of Theorem \ref{thm:main}. For $0<s\leq r$, let
 $$
 \phi_s=\left\{\begin{array}{lr} \phi \text{ in }B_s(x_0),\\
u\text{ in }\R^n\setminus B_s(x_0).
 \end{array}\right.
 $$
 Since $u$ is a viscosity subsolution, $L \phi_s\,(x_0)\leq C$. Now introduce the notation 
 \begin{align*}
\delta(\phi_s, x,y)=&\frac{1}{2}|\phi_s(x)-\phi_s(x+y)|^{p-2}(\phi_s(x)-\phi_s(x+y))\\&+\frac{1}{2}|\phi_s(x)-\phi_s(x-y)|^{p-2}(\phi_s(x)-\phi_s(x-y)),
\end{align*}
$$
 \delta^\pm(\phi_s,x,y) = \max(\pm \delta(\phi_s,x,y),0).
$$
By simply interchanging $y\to-y$ we have
\begin{equation}\label{eq:deltasubsol}
\int_{\R^n}\delta(\phi_s,x_0,y)K(x_0,y)\, dy\leq C,
\end{equation}
since one can easily see that the integral is well defined since $\phi_s$ is $C^2$ near $x_0$. Moreover, 
$$\delta(\phi_{s_2},x_0,y)\leq \delta(\phi_{s_1},x_0,y)\leq \delta(u,x_0,y)\text{ for }s_1< s_2< r,
$$
so that 
$$
\delta^-(u,x_0,y)\leq |\delta(\phi_r,x_0,y)|.
$$
Since $|\delta(\phi_r,x_0,y)K(x,y)|$ is integrable, so is $\delta^-(u,x_0,y)K(x,y)$. In addition, by \eqref{eq:deltasubsol}
$$
\int_{\R^n}\delta^+(\phi_s,x_0,y)K(x_0,y)\, dy\leq \int_{\R^n}\delta^-(\phi_s,x_0,y)K(x_0,y)\, dy+C.
$$
Thus, for $s_1<s_2$
\begin{align}\label{eq:srineq}
\int_{\R^n}\delta^+(\phi_{s_1},x_0,y)K(x_0,y)\, dy&\leq\int_{\R^n}\delta^-(\phi_{s_1},x_0,y)K(x_0,y)\, dy+C\\
&\leq\int_{\R^n}\delta^-(\phi_{s_2},x_0,y)K(x_0,y)\, dy+C<\infty.\nonumber
\end{align}
Since $\delta^+(\phi_s,x_0,y)\nearrow \delta^+(u,x_0,y)$, the monotone convergence theorem implies
$$
\int_{\R^n}\delta^+(\phi_s,x_0,y)K(x_0,y)\, dy \to \int_{\R^n}\delta^+(u,x_0,y)K(x_0,y)\, dy,
$$
and by \eqref{eq:srineq}
\begin{equation}\label{eq:deltaplus}
\int_{\R^n}\delta^+(u,x_0,y)K(x_0,y)\, dy\leq \int_{\R^n}\delta^-(\phi_s,x_0,y)K(x_0,y)\, dy+C<\infty,
\end{equation}
for any $0<s<r$. We conclude that $\delta^+(u,x_0,y)K(x_0,y)$ is integrable. By \eqref{eq:srineq} and the bounded convergence theorem, we can pass to the limit in the right hand side of \eqref{eq:deltaplus} and obtain
$$
\int_{\R^n}\delta (u,x_0,y)K(x_0,y)\, dy=\lim_{s\to 0}\int_{\R^n}\delta (\phi_s,x_0,y)K(x_0,y)\, dy\leq C.
$$
This implies that $Lu\,(x_0)$ exists in the pointwise sense and $Lu\, (x_0)\leq C$.
\end{proof}
\section{H\"older regularity for constant exponents}
In this section we give the proof of our main theorem for the case of constant $s$ and $p$. This is based on Lemma \ref{lem:key}, sometimes referred to as the oscillation lemma. Throughout this section, $L$ denotes an operator of the form in Theorem \ref{thm:main}, i.e.,
$$
Lu\,(x):=\pv \int_{\R^n}|u(x)-u(x+y)|^{p-2}(u(x)-u(x+y))K(x,y)\, dy.
$$
Let us also, by abuse of notation, introduce the function 
$$
\beta(x)=\beta(|x|)=\left((1-|x|^2)^+\right)^2.
$$
The exact form of $\beta$ is not important, we could have chosen any radial function which is $C^2$ and zero outside $B_1$ and non-increasing along rays from the origin.

We start with a couple of auxiliary inequalities. Here $a,b\in \R$.

\begin{lem}\label{lem:pineq1} Let $p\geq 2$. Then
$$
\big||a+b|^{p-2}(a+b)-|a|^{p-2}a\big |\leq (p-1)|b|(|a|+|b|)^{p-2}.
$$   
\end{lem}
\begin{proof}
 We have
 \begin{align*}
 \big||a+b|^{p-2}(a+b)-|a|^{p-2}a\big|&\leq \int_0^{|b|}\Big| \frac{d}{ds} (|a+s|^{p-2}(a+s)\Big|\,ds\\
 &= \int_0^{|b|}(p-1)|a+s|^{p-2}\, ds\\
 &\leq (p-1)|b|(|a|+|b|)^{p-2}.
 \end{align*}
\end{proof}

\begin{lem}\label{lem:pineq2}
Let $p\in (1,2)$. Then
$$
 \big||a+b|^{p-2}(a+b)-|a|^{p-2}a\big|\leq (3^{p-1}+2^{p-1})|b|^{p-1}.
 $$
\end{lem}
\begin{proof}
 We split the proof into two cases.
 
 \noindent {\bf Case 1: $|a|\leq 2|b|$.} Then
 $$
\big| |a+b|^{p-2}(a+b)-|a|^{p-2}a\big|\leq |a+b|^{p-1}+|a|^{p-1}\leq (3^{p-1}+2^{p-1})|b|^{p-1}.
 $$
  \noindent {\bf Case 2: $|a|> 2|b|$.} Then for $|s|\leq |b|$
  $$  
  |a+s|\geq |a|-|s|>2|b|-|b|=|b|,
  $$
  so that
  $$
\big| |a+b|^{p-2}(a+b)-|a|^{p-2}a\big|\leq \int_0^{|b|}(p-1)|a+s|^{p-2} \, ds \leq (p-1)|b|^{p-1}.
 $$
 Since $p-1\leq 3^{p-1}+2^{p-1}$, this concludes the proof.
\end{proof}

\begin{lem}\label{lem:pest}
Let $p\geq 2$ and assume $a+b\geq 0$. Then
$$
|a+b|^{p-2}(a+b)\leq 2^{p-2}(|a|^{p-2}a+|b|^{p-2}b).
$$ 
\end{lem}
\begin{proof} The inequality is trivial for $p=2$ so we assume $p>2$. Since $a+b\geq 0$, $|a|^{p-2}a+|b|^{p-2}b\geq 0$. Without loss of generality we can assume $a>0$ and define $t=b/a$. The statement of the lemma is then equivalent to 
 $$ 
 |1+t|^{p-2}(1+t)\leq 2^{p-2}(1+|t|^{p-2}t), \text{ for }t\geq -1.
 $$
This is trivially true for $t=-1$. Hence we are lead to study the function 
$$
f(t):=\frac{|1+t|^{p-2}(1+t)}{1+|t|^{p-2}t}, \text{ for }t> -1.
$$
We find that $f$ has critical points at $t=1$ and $t=0$. In addition, 
$$
f(1)=2^{p-2}, \lim_{t\searrow -1} f(t)=0, f(0)=1, \lim_{|t|\to \infty} f(t)=1.
$$
We conclude that $f(t)\leq 2^{p-2}$ for all $t\geq -1$, and the result follows.
\end{proof}

Below we prove that a kernel $K$ behaving like $y^{-n-sp}$ satisfies certain inequalities that might look strange at a first glance, but they are exactly the ones that will appear in the proof of our key lemma later.
\begin{prop}\label{prop:fixp}
Assume $K$ satisfies $K(x,y)=K(x,-y)$ and there exist $\Lambda\geq \lambda>0$, $M>0$ and $\gamma>0$ such that
\begin{align*}
\frac{\lambda}{|y|^{n+sp}}\leq & K(x,y)\leq \frac{\Lambda}{|y|^{n+sp}}, \text{ for } y\in B_2,x\in B_2,\\ 
0\leq & K(x,y) \leq \frac{M}{|y|^{n+\gamma}}, \text{ for } y\in \R^n\setminus B_\frac14,x\in B_2, 
\end{align*}
where $s\in (0,1)$ and $p\in (1,\infty)$. In the case $p<2$ we require additionally $p>1/(1-s)$. Then for any $\delta>0$ there are $1/2\geq k>0$ and $\eta>0$ such that for $p\in (2,\infty)$
\begin{align}\label{eq:kassp2}
&2^{p-2}k^{p-1}\pv \int_{x+y\in B_1}|\beta(x)-\beta(x+y)|^{p-2}(\beta(x)-\beta(y+x))K(x,y)\, dy  \nonumber \\
&+2^{p-2}\int_{y\in \R^n\setminus B_\frac14}|k\beta(x)+2(|8y|^ \eta-1)|^{p-1} K(x,y)\, dy\\ \nonumber
&+2^{p-1}\int_{y\in \R^n\setminus B_\frac14}(|8y|^ \eta-1)^{p-1} K(x,y)\, dy<2^{1-p}\inf_{A\subset B_2,|A|>\delta}\int_A K(x,y)\,d y
\end{align}  
and for $p\in (1/(1-s),2)$
\begin{align}\label{eq:kassp1}
(3^{p-1}+2^{p-1})k^{p-1}\int_{\R^n}|\beta(x)-\beta(x+y)|^{p-1}K(x,y)\, dy\\ \nonumber
+2^{p-1}\int_{\R^n\setminus B_\frac14}(|8y|^\eta-1)^{p-1}K(x,y)\,dy<2^{1-p}\inf_{A\subset B_2,|A|>\delta}\int_A K(x,y)\,d y,
\end{align}
for any $x\in B_{3/4}$.
Here $k$ and $\eta$ depend on $\lambda,\Lambda,M,p,s,\gamma$ and $\delta$.
\end{prop}
\begin{proof}
The proof is split into two different cases.\\
\noindent {\bf Case 1: $p>2$}\\
The first term in the left hand side of \eqref{eq:kassp2} reads
\begin{align*}
&2^{p-2}k^{p-1}\pv \int_{x+y\in B_1}|\beta(x)-\beta(x+y)|^{p-2}(\beta(x)-\beta(y+x))K(x,y)\, dy\\
=&2^{p-2}k^{p-1}\pv \int_{x+y\in B_1,y\not\in B_\frac14}|\beta(x)-\beta(x+y)|^{p-2}(\beta(x)-\beta(y+x))K(x,y)\, dy\\
&+2^{p-2}k^{p-1}\pv \int_{y\in B_\frac14}|\beta(x)-\beta(x+y)|^{p-2}(\beta(x)-\beta(y+x))K(x,y)\, dy\\
&=I_1+I_2.
\end{align*}
Since $\beta$ is uniformly bounded by a constant $C$, we can, using the upper bound on $K$ outside $B_{1/4}$, obtain
\begin{equation}\label{eq:case1est1a}
|I_1|\leq |2kC|^{p-1}\int_{\R^n\setminus B_\frac14} K(x,y)\,dy\leq |2kC|^{p-1} M\int_{\R^n\setminus B_\frac14} \frac{dy}{|y|^{n+\gamma}},  
\end{equation}
which is finite and converges to zero as $k\to 0$.

For $I_2$ we proceed as follows
\begin{align*}
I_2&=2^{p-2}k^{p-1}\pv \int_{y\in B_\frac14}|\beta(x)-\beta(x+y)|^{p-2}(\beta(x)-\beta(y+x))K(x,y)\, dy\\
&=2^{p-3}k^{p-1}\pv \int_{y\in B_\frac14}|\beta(x)-\beta(x+y)|^{p-2}(\beta(x)-\beta(y+x))K(x,y)\,dy \\
&+2^{p-3}k^{p-1}\pv \int_{y\in B_\frac14}|\beta(x)-\beta(-y+x)|^{p-2}(\beta(x)-\beta(-y+x))K(x,y)\,dy.
\end{align*}
Introducing the notation
$$
F=-(\beta(x)-\beta(x-y)),\quad G=(\beta(x)-\beta(x-y))+(\beta(x)-\beta(x+y)), 
$$
$I_2$ can be written as
\begin{align*}
&2^{p-3}k^{p-1}\int_{y\in B_\frac14}\left(|F+G|^{p-2}(F+G)-|F|^{p-2}F\right) K(x,y) \,dy\\
&\leq 2^{p-3}k^{p-1}(p-1)\int_{y\in B_\frac14}|G|(|F|+|G|)^{p-2}K(x,y)\,dy,
\end{align*}
by Lemma \ref{lem:pineq1}. Since $\beta$ is $C^2$, $|F|\leq C|y|$ and $|G|\leq C|y|^2$. Invoking the upper bound on $K$ in $B_2$ yields the estimate
\begin{equation}\label{eq:case1est1}
I_2\leq C^{p-1}2^{p-3}k^{p-1}(p-1)\Lambda\int_{y\in B_\frac14}|y|^{p-n-sp}\, dy\leq \frac{C^{p-1}2^{p-3}k^{p-1}(p-1)\Lambda \left(\frac14\right)^{p(1-s)}}{p(1-s)},
\end{equation}
where $C$ only depends on the $C^2$-norm of $\beta$, which is fixed. Clearly the left hand side of \eqref{eq:case1est1} goes to zero as $k\to 0$.

For the rest of the terms in the left hand side we observe first that if $\eta<\gamma/(p-1)$ then from the upper bound on $K$ outside $B_{1/4}$
\begin{equation}\label{eq:case1est2}
\int_{\R^n\setminus B_\frac14}\left(|8y|^\eta-1\right)^{p-1} K(x,y)\, dy\leq M\int_{\R^n\setminus B_\frac14}\left(|8y|^\eta-1\right)^{p-1} \frac{dy }{|y|^{n+\gamma}}, 
\end{equation}
which is uniformly bounded and tends to zero as $\eta \to 0$, by the dominated convergence theorem.

In addition, since $\beta$ is uniformly bounded by some constant $C>0$ we have
\begin{equation}\label{eq:case1est3}
\int_{\R^n\setminus B_\frac14}|k\beta(x)|^{p-1} K(x,y)\,dy\leq k^{p-1}C^{p-1}M\int_{\R^n\setminus B_\frac14} \frac{dy}{|y|^{n+\gamma}},  
\end{equation}
which is finite and converges to zero as $k\to 0$, where we again have used the upper bound on $K$ outside $B_{1/4}$.

Thus, if we choose $\eta$ and $k$ small enough (depending on $\Lambda$, $M$, $p$, $s$ and $\gamma$) we can make all the terms in the left hand side as small as desired.

Now we turn our attention to the right hand side. We have, due to the lower bound on $K$ in $B_2$
$$
2^{1-p}\inf_{A\subset B_2,|A|>\delta}\int_{A}K(x,y)\,dy \geq \frac{2^{1-p}\lambda \delta}{2^{n+sp}}.
$$
Then it is clear that we can choose $\eta$ and $k$, depending only on $\lambda$, $\Lambda$, $M$, $p$, $s$, $\gamma$ and $\delta$, so that the left hand side is larger than the right hand side.

\noindent {\bf Case 2: $1/(1-s)<p<2$}\\
The only difference from the case $p>2$ is the first term in the left hand side. We need to show that for $k$ small enough, the term 
$$(3^{p-1}+2^{p-1})k^{p-1}\int_{\R^n}|\beta(x)-\beta(x+y)|^{p-1}K(x,y)\, dy,$$
is small. We split the integral into two parts, one in $B_1$ and one in $\R^n\setminus B_1$. We have $|\beta(x)-\beta(x+y)|\leq C|y|$ for $y\in B_1$ and $|\beta(z)|\leq C$ for all $z\in \R^n$. Hence, 
\begin{align}
&(3^{p-1}+2^{p-1})k^{p-1}\int_{B_1}|\beta(x)-\beta(x+y)|^{p-1}K(x,y)\, dy\nonumber \\\label{eq:case2est1}
&\leq \Lambda C^{p-1}(3^{p-1}+2^{p-1})k^{p-1}\int_{B_1}|y|^{p-1-n-sp}\,dy\\\nonumber
&\leq \Lambda C^{p-1}(3^{p-1}+2^{p-1})k^{p-1}\frac{1}{p(1-s)-1},
\end{align}
where we have used the upper bound on $K$ in $B_2$. For the part outside $B_1$ we have
\begin{align}
&(3^{p-1}+2^{p-1})k^{p-1}\int_{\R^n\setminus B_1}|\beta(x)-\beta(x+y)|^{p-1}K(x,y)\, dy\nonumber \\\label{eq:case2est2}
&\leq C^{p-1}M(3^{p-1}+2^{p-1})k^{p-1}\int_{\R^n\setminus B_1}\frac{dy}{|y|^{n+\gamma}}\\\nonumber
&\leq C^{p-1}M(3^{p-1}+2^{p-1})k^{p-1}\gamma^{-1},
\end{align}
from the upper bound on $K$ outside $B_{1/4}$.
By choosing $k$ small (depending on $\Lambda$, $M$, $p$, $s$, $\gamma$) we can make both of these terms as small as desired. Hence, the result follows as in the case $p>2$.

\end{proof}

\begin{rem}
We remark that in the proof above, nothing would change if the exponents would depend on $x$, since $x$ is a fixed point. This is important later when we redo the proof for the case of variable exponents.
\end{rem}

The lemma below is the core of this paper. The proof is an adaptation of the proof of Lemma 4.1 in \cite{Sil06}.

\begin{lem}\label{lem:key} Assume the hypotheses of Proposition \ref{prop:fixp}. Suppose 
\begin{align*}
Lu\leq 0\text{ in }B_1,\\
u\leq 1\text{ in }B_1,\\
u(x)\leq 2|2x|^\eta-1\text{ in }\R^n\setminus B_1,\\
|B_1\cap \{u\leq 0\}|>\delta,
\end{align*}
where $\eta$ is as in Proposition \ref{prop:fixp}.
Then $u\leq 1-\theta$ in $B_{1/2}$, where $\theta =\theta(\lambda,\Lambda,M,p,s,\gamma,\delta)>0$.
\end{lem}
\begin{proof} We argue by contradiction. Let
$$  
\theta=k\left(\beta(1/2)-\beta(3/4)\right),
$$
where $k$ is as in Proposition \ref{prop:fixp}. If there is $x_0\in B_{1/2}$ such that $u(x_0)>1-\theta$, then 
$$
u(x_0)+k\beta(1/2)>1+k\beta(3/4).
$$
Moreover, for any $y\in B_1\setminus B_{3/4}$ there holds
$$
u(x_0)+k\beta(x_0)>u(x_0)+k\beta(1/2)>1+k\beta(3/4)\geq u(y)+k\beta(y).
$$
Hence, the maximum of $u+k\beta$ in $B_1$ is attained inside $B_{3/4}$ and it is strictly larger than 1. Suppose that the maximum is attained at the point $x$.

The rest of the proof is devoted to estimating $L(u+k\beta)\,(x)$ from above and from below in order to obtain a contradiction with Proposition \ref{prop:fixp}. At this point, we remark that $-k\beta+(u+k\beta)(x)$ touches $u$ from above at $x$. Hence, by Proposition \ref{prop:pw}, $Lu\,(x)\leq 0$ in the pointwise sense.


We first estimate $L(u+k\beta)\, (x)$ from below. We split the integrals into two parts and write
\begin{align*}
L(u+k\beta)\, (x)&=\pv \int_{x+y\in B_1}+\int_{x+y\not\in B_1}\\
&=\lim_{r\to 0}\int_{x+y\in B_1,y\not\in B_r}+\int_{x+y\not\in B_1}=\lim_{r\to 0} I_r+I_2,
\end{align*}
where there is no need for the principal value in the second integral, since $x\in B_{3/4}$. Using that $u(x)+k\beta(x)>1$ is the maximum of $u+k\beta$ in $B_1$ we see that the integrand in $I_r$ is non-negative and we have the estimate
\begin{align*}
 I_r\geq  \int_{A_0}(1-k\beta(x+y))^{p-1}K(x,y)\, dy, 
\end{align*}
where 
$$A_0=\{x+y\in B_1,\quad u(x+y)\leq 0\}.$$
Since $\beta \leq 1$ and $k\leq 1/2$ we conclude
$$
I_r\geq \frac{1}{2^{p-1}}\inf_{A_0\subset B_2,|A_0|>\delta}\int_{A_0}K(x,y)\,d y.
$$
Now we estimate $I_2$ from below. Using that $u(x)+k\beta(x)>1$ and $u(z)\leq 2|2z|^\eta-1$ for $z\in \R^n\setminus B_1$ and $\beta=0$ in $\R^n\setminus B_1$, we have
\begin{align*}
 I_2&\geq \int_{x+y\not\in B_1} 2^{p-1}\Big|1-|2(x+y)|^\eta\Big|^{p-2}(1-|2(x+y)|^\eta) K(x,y)\, dy\\
 &\geq 2^{p-1}\int_{y\not\in B_\frac14}\Big|1-\Big|2\left(|y|+\frac34\right)\Big|^\eta\Big|^{p-2}\left(1-\Big|2\left(|y|+\frac34\right)\Big|^\eta\right)K(x,y)\, dy\\
 &\geq -2^{p-1}\int_{y\not\in B_\frac14}(|8y|^\eta-1)^{p-1} K(x,y)\, dy.
\end{align*}
Adding the two estimates together we can summarize
\begin{align}\label{eq:Lfrombelow}
&L(u+k\beta)\,(x)\geq \\
\nonumber & \frac{1}{2^{p-1}}\inf_{A_0\subset B_2,|A_0|>\delta}\int_{A_0}K(x,y)\,d y-2^{p-1}\int_{y\not\in B_\frac14}(|8y|^\eta-1)^{p-1} K(x,y)\, dy.
\end{align}

The next step is to estimate $L(u+k\beta)\, (x)$ from above. This part of the proof is split into two cases: $p\geq 2$ and $p<2$. 

\noindent {\bf Case 1: $p\geq 2$}\\
Again we split the integral defining $L(u+k\beta)\, (x)$ into two parts
$$
L(u+k\beta)\, (x)=\pv \int_{x+y\in B_1}+\int_{x+y\not\in B_1}:=I_1+I_2,
$$
where again, there is no need for the principal value in the second integral. We first treat $I_1$ by noting that when $x+y\in B_1$, we know
$$
u(x)+k\beta(x)-u(x+y)-k\beta(x+y)\geq 0,
$$
recalling that $u+k\beta$ attains its maximum (in $B_1$) at $x$.

From Lemma \ref{lem:pest}
\begin{align*}
 |u(x)&-u(x+y)+ k\beta(x)-k\beta(x+y)|^{p-2}(u(x)-u(x+y)+k\beta(x)-k\beta(x+y))\leq \\
&2^{p-2}|u(x)-u(x+y)|^{p-2}(u(x)-u(x+y))\\
+&2^{p-2}|k\beta(x)-k\beta(x+y)|^{p-2}(k\beta(x)-k\beta(x+y)).
\end{align*}
Hence, 
\begin{align*}
I_1&\leq 2^{p-2}\pv\int_{x+y\in B_1} |u(x)-u(x+y)|^{p-2}(u(x)-u(x+y))K(x,y)\,dy  \\
&+ 2^{p-2}k^{p-1}\pv\int_{x+y\in B_1} |\beta(x)-\beta(x+y)|^{p-2}(\beta(x)-\beta(x+y)) K(x,y)\,dy.
\end{align*}
Now we turn our attention to $I_2$. We note that when $x+y\not\in B_1$, we cannot apply Lemma \ref{lem:pest} directly, but we still have from the hypothesis
$$
u(x)+k\beta(x)>1,\quad u(x+y)+k\beta(x+y)\leq 2|2(x+y)|^\eta-1.
$$
In other words, 
$$
u(x)-u(x+y)+k\beta(x)-k\beta(x+y)>2(1-|2(x+y)|^\eta).
$$
By adding the term $2(|2(x+y)|^\eta-1)>0$ to the the expression, we increase the integrand, and we also make the integrand non-negative so that we can, once more, apply Lemma \ref{lem:pest}. It follows that
\begin{align*}
I_2&\leq \int_{x+y\not\in B_1}|u(x)-u(x+y)+k\beta(x)-k\beta(x+y)+2(|2(x+y)|^\eta-1)|^{p-2}\times \\
&(u(x)-u(x+y)+k\beta(x)-k\beta(x+y)+2(|2(x+y)|^\eta-1))K(x,y)\, dy\\
&\leq 2^{p-2}\int_{x+y\not\in B_1} |u(x)-u(x+y)|^{p-2}(u(x)-u(x+y))K(x,y)\,dy\\
&+2^{p-2}\int_{x+y\not\in B_1}|k\beta(x)-k\beta(x+y)+2(|2(x+y)|^\eta-1)|^{p-2}\times \\
&(k\beta(x)-k\beta(x+y)+2(|2(x+y)|^\eta-1))K(x,y)\, dy.
\end{align*}
Adding the estimates for $I_1$ and $I_2$ together we arrive at
\begin{align}\nonumber 
&L(u+k\beta)\,(x)\leq 2^{p-2}Lu\, (x)\\\nonumber 
&+2^{p-2}k^{p-1}\pv\int_{x+y\in B_1} |\beta(x)-\beta(x+y)|^{p-2}(\beta(x)-\beta(x+y)) K(x,y)\,dy\\\nonumber 
&+2^{p-2}\int_{x+y\not\in B_1}|k\beta(x)-k\beta(x+y)+2(|2(x+y)|^\eta-1)|^{p-2}\times\\ \label{eq:Lfromabove1}ß
&(k\beta(x)-k\beta(x+y)+2(|2(x+y)|^\eta-1))K(x,y)\, dy\\\nonumber 
&\leq 2^{p-2}k^{p-1}\pv\int_{x+y\in B_1} |\beta(x)-\beta(x+y)|^{p-2}(\beta(x)-\beta(x+y)) K(x,y)\,dy\\\nonumber 
&+2^{p-2}\int_{x+y\not\in B_1}|k\beta(x)-k\beta(x+y)+2(|2(x+y)|^\eta-1)|^{p-1}K(x,y)\, dy,
\end{align}
since $Lu\,(x)\leq 0$.

\noindent {\bf Case 2: $\frac{1}{1-s}<p<2$}\\ 
From Lemma \ref{lem:pineq2}
\begin{align*}
|u(x)-u(x+y)+k\beta(x)-k\beta(x+y)|^{p-2}(u(x)-u(x+y)+k\beta(x)-k\beta(x+y))\leq \\
|u(x)-u(x+y)|^{p-2}(u(x)-u(x+y))+(3^{p-1}+2^{p-1})k^{p-1}|\beta(x)-\beta(x+y)|^{p-1}
\end{align*}
from which it follows that
\begin{align}
L(u+k\beta)\,(x)&\leq Lu\,(x)+k^{p-1}(3^{p-1}+2^{p-1})\int_{\R^n}|\beta(x)-\beta(x+y)|^{p-1} K(x,y)\,d y\label{eq:Lfromabove2}\\\nonumber
&\leq k^{p-1}(3^{p-1}+2^{p-1})\int_{\R^n}|\beta(x)-\beta(x+y)|^{p-1} K(x,y)\,d y.
\end{align}

Finally, we arrive at a contradiction by observing that \eqref{eq:Lfrombelow} combined with either \eqref{eq:Lfromabove1} or \eqref{eq:Lfromabove2} results in a contradiction with \eqref{eq:kassp2} or \eqref{eq:kassp1} in Proposition \ref{prop:fixp}.
\end{proof}
Once the lemma above is established, the proof of the H\"older regularity is standard. We follow the lines of the proof of Theorem 5.1 in \cite{Sil06}.

\begin{proof}[~Proof of Theorem \ref{thm:main}]We first rescale $u$ by the factor
$$
\frac{1}{2\|u\|_{L^\infty(\R^n)}}.
$$
Then the new $u$ satisfies
$$
L u =0\text{ in }B_1,\quad \osc_{\R^n}u \leq 1.
$$
We will now show that for $j=0,1,\ldots$
$$
\osc_{B_{2^{-j}(x_0)}} u\leq 2^{-j\alpha}, \text{ for any }x_0\in B_1, 
$$
where $\alpha$ is chosen so that 
$$
\frac{2-\theta}{2}\leq 2^{-\alpha}\text{ and } \alpha\leq \eta, 
$$
where $\theta$ is from Lemma \ref{lem:key} and $\eta$ is from Proposition \ref{prop:fixp}, with  $\delta=|B_1|/2$. This will imply the desired result with $C=2^{\alpha}$.

In what follows we will find constants $a_j$ and $b_j$ so that
\begin{equation}\label{eq:akbk}
b_j\leq u\leq a_j\text{ in  }B_{2^{-j}(x_0)},\quad |a_j-b_j|\leq 2^{-j\alpha}.
\end{equation}
We construct these by induction. For $j\leq 0$, \eqref{eq:akbk} holds true with $b_j=\inf_{\R^n} u $ and $a_j=b_j+1$.

Assume \eqref{eq:akbk} holds for all $j\leq k$. We need to construct $a_{k+1}$ and $b_{k+1}$. Put $m=(a_k+b_k)/2$. Then
$$
|u-m|\leq 2^{-k\alpha-1}\text{ in  $B_{2^{-k}}(x_0)$.}
$$
Let 
$$
v(x)=2^{\alpha k+1}(u(2^{-k}x+x_0)-m).
$$
Then 
$$
\pv \int_{\R^n}|v(x)-v(x+y)|^{p-2}(v(x)-v(x+y))K_{x_0,2^{-k}}(x,y)\, dy= 0 \text{ in }B_1
$$
and
$$
 |v|\leq 1 \text{ in }B_1,
$$
where 
$$
K_{x_0,2^{-k}}(x,y)=2^{-k(n+sp)}K(2^{-k}x+x_0,2^{-k}y),$$
which satisfies the same assumptions as $K$ itself. We also remark for $|y|>1$ such that $2^\ell\leq |y|\leq 2^{\ell+1}$ we have
\begin{align*}
v(y)= 2^{\alpha k+1}(u(2^{-k}y+x_0)-m)&\leq 2^{\alpha k+1}(a_{k-\ell-1}-m)\\
&\leq 2^{\alpha k+1}(a_{k-\ell-1}-b_{k-\ell-1}+b_{k}-m)\\
&\leq 2^{\alpha k+1}(2^{-\alpha(k-\ell-1)}-\frac12 2^{-k\alpha})\\
&\leq 2^{1+\alpha(\ell+1)}-1\leq 2|2y|^\alpha-1\\
&\leq 2|2y|^\eta-1,
\end{align*}
where we have used that \eqref{eq:akbk} holds for $j\leq k$. Suppose now that \mbox{$|\{v\leq 0\}\cap B_1|\geq |B_1|/2$} (if not we would apply the same procedure to $-v$). Then $v$ satisfies all the assumptions of Lemma \ref{lem:key}, with $\delta =|B_1|/2$ and we obtain
$$
v(x)\leq 1-\theta\text{ in }B_\frac12, 
$$
where $\theta=\theta(\lambda,\Lambda,M,p,s,\gamma)$, since $\delta$ is fixed. Scaling back to $u$ this yields
\begin{align*}
u(x)&\leq 2^{-1-\alpha k}(1-\theta)+m\leq 2^{-1-k\alpha}(1-\theta)+\frac{a_k+b_k}{2}\\
&\leq b_k+2^{-1-\alpha k}(1-\theta)+2^{-1-\alpha k}\\
&\leq b_k+2^{-\alpha (k+1)}
\end{align*}
by our choice of $\alpha$. Hence, if we let $b_{k+1}=b_k$ and $a_{k+1}=b_k+2^{-\alpha(k+1)}$ we obtain \eqref{eq:akbk} for the step $j=k+1$ and the induction is complete.
\end{proof} 
\section{Variable exponents}
In this section we show that our results also apply to the case when both $p$ and $s$ vary with $x$. In particular we prove Theorem \ref{thm:main2}. Throughout this section $L$ denotes the operator
$$
Lu\,(x):=\pv \int_{\R^n}|u(x)-u(x+y)|^{p(x)-2}(u(x)-u(x+y))K(x,y)\, dy.
$$
We follow the same strategy as in the case of constant exponents and prove slightly modified versions of Proposition \ref{prop:fixp} and Lemma \ref{lem:key}. The proof of H\"older continuity is then similar.

\begin{prop}\label{prop:varp}
Assume $K$ satisfies $K(x,y)=K(x,-y)$ and there exist $\Lambda\geq \lambda>0$, $M>0$ and $\gamma>0$ such that
\begin{align*}
\frac{\lambda}{|y|^{n+s(x)p(x)}}\leq & K(x,y)\leq \frac{\Lambda}{|y|^{n+s(x)p(x)}}, \text{ for } y\in B_2,x\in B_2,\\ 
0\leq & K(x,y) \leq \frac{M}{|y|^{n+\gamma}}, \text{ for } y\in \R^n\setminus B_\frac14,x\in B_2, 
\end{align*}
where $0<s_0<s(x)<s_1<1$ and $1<p_0<p(x)<p_1<\infty$. In the case $p(x)<2$ we require additionally that there is $\tau>0$ such that
$$
p(x)(1-s(x))-1>\tau .$$ 
Then for any $\delta>0$ there are $1/2 \geq k>0$ and $\eta>0$ such that for $p\in (2,\infty)$
\begin{align}\label{eq:varkassp2}
&2^{p(x)-2}k^{p(x)-1}\pv \int_{x+y\in B_1}|\beta(x)-\beta(x+y)|^{p(x)-2}(\beta(x)-\beta(y+x))K(x,y)\, dy  \nonumber \\
&+2^{p(x)-2}\int_{y\in \R^n\setminus B_\frac14}|k\beta(x)+2((|8y|^ \eta-1)|^{p(x)-1} K(x,y)\, dy\\ \nonumber
&+2^{p(x)}\int_{y\in \R^n\setminus B_\frac14}((|8y|^ \eta-1)|^{p(x)-1} K(x,y)\, dy<2^{1-p(x)}\inf_{A\subset B_2,|A|>\delta}\int_A K(x,y)\,d y
\end{align}  
and for $p(x)\in (1/(1-s),2)$
\begin{align}\label{eq:varkassp1}
(3^{p(x)-1}+2^{p(x)-1})k^{p(x)-1}\int_{\R^n}|\beta(x)-\beta(x+y)|^{p(x)-1}K(x,y)\, dy\\ \nonumber
+2^{p(x)}\int_{\R^n\setminus B_\frac14}(|8y|^\eta-1)^{p(x)-1}K(x,y)\,dy<2^{1-p(x)}\inf_{A\subset B_2,|A|>\delta}\int_A K(x,y)\,d y,
\end{align}
for any $x\in B_{3/4}$.
Here $k$ and $\eta$ depend on $\lambda,\Lambda,M,p_0,p_1,s_0,s_1,\gamma,\tau$ and $\delta$.
\end{prop} 
\begin{proof}

We point out the differences to the proof of Proposition \ref{prop:varp} and briefly explain how they can be dealt with.

\noindent{\bf Case 1: $p(x)\geq 2$}\\
By the exact same computation as in \eqref{eq:case1est1a}, \eqref{eq:case1est1}, \eqref{eq:case1est2} and \eqref{eq:case1est3} in the proof of Proposition \ref{prop:fixp} (since the computation is made for a fixed $x$), we can conclude that the left hand side is bounded by
\begin{align*}
 |2kC|^{p-1} M\int_{\R^n\setminus B_\frac14} \frac{dy}{|y|^{n+\gamma}}+\frac{C^{p(x)-1}2^{p(x)-3}k^{p(x)-1}(p(x)-1)\Lambda \left(\frac14\right)^{p(x)(1-s(x))}}{p(x)(1-s(x))}
 \end{align*}
 plus terms involving the quantities
 \begin{align*}
M\int_{\R^n\setminus B_\frac14}\left(|8y|^\eta-1\right)^{p(x)-1} \frac{dy }{|y|^{n+\gamma}}
\end{align*}
and
\begin{align*}
k^{p(x)-1}C^{p(x)-1}2^{p(x)-1}M\int_{\R^n\setminus B_\frac14} \frac{dy}{|y|^{n+\gamma}}.  
\end{align*}
Due to the assumptions on $p$ and $s$, the terms are all uniformly bounded. Thus, if we choose $\eta$ and $k$ small enough (depending on $\Lambda$, $M$, $p_0$, $p_1$, $s_0$, $s_1$ and $\gamma$) we can make all the terms in the left hand side as small as desired.

For the right hand side, we again have
$$
2^{1-p(x)}\inf_{A\subset B_2,|A|>\delta}\int_{A}K(x,y)\,dy \geq \frac{2^{1-p(x)}\lambda \delta}{2^{n+s(x)p(x)}}\geq \frac{2^{1-p_1}\lambda \delta}{2^{n+s_1p_1}}.
$$
Then it is clear that we can choose $\eta$ and $k$, depending only on $\lambda$, $\Lambda$, $M$, $p_0$, $p_1$, $s_0$, $s_1$, $\gamma$ and $\delta$, so that the left hand side is larger than the right hand side.

\noindent{\bf Case 2: $1/(1-s(x))<p(x)<2$}\\
We can again estimate the left hand side by
\begin{align*}
&\Lambda C^{p(x)-1}(3^{p(x)-1}+2^{p(x)-1})k^{p(x)-1}\frac{1}{p(x)(1-s(x))-1}\\&+C^{p(x)-1}M(3^{p(x)-1}+2^{p(x)-1})k^{p(x)-1}\gamma^{-1},
\end{align*}
as in \eqref{eq:case2est1} and \eqref{eq:case2est2}. By choosing $k$ small (depending on $\Lambda$, $M$, $p_0$, $p_1$, $\tau$, $\gamma$) we can make both these terms as small as desired. The result follows also in this case.
\end{proof}

\begin{lem}\label{lem:key2} Assume the hypotheses of Proposition \ref{prop:varp}. Suppose 
\begin{align*}
Lu\leq \e\text{ in }B_1,\\
u\leq 1\text{ in }B_1,\\
u(x)\leq 2|2x|^\eta-1\text{ in }\R^n\setminus B_1,\\
|B_1\cap \{u\leq 0\}|>\delta,
\end{align*}
where $\eta$ is as in Proposition \ref{prop:varp} and
\begin{align*}
\e=\min(2,2^{p(x)-1})\int_{y\not\in B_\frac14}(|8y|^\eta-1)^{p(x)-1} K(x,y)\, dy.
\end{align*}
Then $u\leq 1-\theta$ in $B_{1/2}$, where $\theta = \theta(\lambda,\Lambda,M,p_0,p_1,s_0,s_1,\gamma,\tau,\delta)>0$.
\end{lem}
\begin{proof}
The first part of the proof is exactly the same as the one of Proposition \ref{prop:fixp}. Then it comes to estimating $L(u+k\beta)\, (x)$ from below. Since $x$ is a fixed point throughout all the calculations, we obtain as in \eqref{eq:Lfrombelow}
\begin{align}\label{eq:Lfrombelowvar}
&L(u+k\beta)\,(x)\geq \\
\nonumber & \frac{1}{2^{p(x)-1}}\inf_{A_0\subset B_2,|A_0|>\delta}\int_{A_0}K(x,y)\,d y-2^{p(x)-1}\int_{y\not\in B_\frac14}(|8y|^\eta-1)^{p(x)-1} K(x,y)\, dy.
\end{align}	
The next step is then to estimate  $L(u+k\beta)\, (x)$ from above. We obtain almost the same estimate as in \eqref{eq:Lfromabove1} and \eqref{eq:Lfromabove2}. The difference is that we instead of $Lu\,(x)\leq 0$ use $Lu\,(x)\leq \e$ and obtain an extra term
$$
\min(2,2^{p(x)-1})\int_{y\not\in B_\frac14}(|8y|^\eta-1)^{p(x)-1} K(x,y)\, dy.
$$
Hence the estimate reads in the two different cases:\\
\noindent{\bf Case 1: $p(x)\geq 2$}
\begin{align}\nonumber 
&L(u+k\beta)\,(x) \\&\label{eq:Lfromabove1var}  \leq 2^{p(x)-2}k^{p(x)-1}\pv\int_{x+y\in B_1} |\beta(x)-\beta(x+y)|^{p(x)-2}(\beta(x)-\beta(x+y)) K(x,y)\,dy\\
&+2^{p(x)-2}\int_{x+y\not\in B_1}|k\beta(x)-k\beta(x+y)+2(|2(x+y)|^\eta-1)|^{p(x)-1}K(x,y)\, dy  \nonumber 
\\
&+2^{p(x)-1}\int_{y\not\in B_\frac14}(|8y|^\eta-1)^{p(x)-1} K(x,y)\, dy .\nonumber 
\end{align}
 \noindent{\bf Case 2: $1/(1-s(x))<p(x)<2$}
 \begin{align}
L(u+k\beta)\,(x)&\leq k^{p(x)-1}(3^{p(x)-1}+2^{p-1})\int_{\R^n}|\beta(x)-\beta(x+y)|^{p(x)-1} K(x,y)\,d y \label{eq:Lfromabove2var}\\
&\nonumber +2^{p(x)-1}\int_{y\not\in B_\frac14}(|8y|^\eta-1)^{p(x)-1} K(x,y)\, dy.
\end{align}
The combination of \eqref{eq:Lfrombelowvar} with either \eqref{eq:Lfromabove1var} or \eqref{eq:Lfromabove2var} is a contradiction to \eqref{eq:varkassp2} or \eqref{eq:varkassp1}.
\end{proof}
\begin{proof}[Proof of Theorem \ref{thm:main2}]
The proof is very similar to the proof of Theorem \ref{thm:main}. We first rescale $u$ by the factor
$$
\left(2\|u\|_{L^\infty(\R^n)}+2^{\frac{p_1-1}{p_0-1}}\max\Big\{\left(\frac{\|f\|_{L^\infty(B_2)}}{\e}\right)^\frac{1}{p_0-1},\left(\frac{\|f\|_{L^\infty(B_2)}}{\e}\right)^\frac{1}{p_1-1}\Big\}\right)^{-1},
$$
where $\e$ is chosen as in Lemma \ref{lem:key2} with $\delta=|B_1|/2$. Then one readily verifies that
$$
L u =\tilde f\text{ in }B_2,\quad \|\tilde f\|_{L^\infty(B_2)}\leq \frac{\e}{2^{p_1-1}},\quad \osc_{\R^n}u \leq 1.
$$
Next we proceed as before: we find $a_j$ and $b_j$ such that
\begin{equation}\label{eq:akbkvar}
b_j\leq u\leq a_j\text{ in  }B_{2^{-j}(x_0)},\quad |a_j-b_j|\leq 2^{-j\alpha},
\end{equation}
where we require from $\alpha$ that
$$
\frac{2-\theta}{2}\leq 2^{-\alpha},  \alpha\leq \eta \text{ and } \alpha\leq \frac{s_0p_0}{p_1-1},
$$
where $\beta$ is from Lemma \ref{lem:key2} and $\eta$ from Proposition \ref{prop:varp}, with $\delta =|B_1|/2$. As before, \eqref{eq:akbkvar} is satisfied for $j\leq 0$ with the choice $b_j=\inf_{\R^n} u$ and $a_j=b_j+1$. Now, given that \eqref{eq:akbkvar} holds for $j\leq k$ we construct $a_{k+1}$ and $b_{k+1}$. Define
$$
v(x)=2^{\alpha k+1}(u(2^{-k}x+x_0)-m),\quad  \text{ with } m=\frac{a_k+b_k}{2}.
$$
Then
\begin{align*}
&\pv \int_{\R^n}|v(x)-v(x+y)|^{p(x)-2}(v(x)-v(x+y))K_{x_0,2^{-k}}(x,y)\, dy \\
&=2^{(\alpha k+1)(p(2^{-k}x+x_0)-1)-k(s(2^{-k}x+x_0)p(2^{-k}x+x_0))}\tilde f\text{ in }B_1,
\end{align*}
and
$$
 |v|\leq 1 \text{ in }B_1.
$$
As before,
$$
K_{x_0,2^{-k}}(x,y)=2^{-k(n+s(2^{-k}x+x_0)p(2^{-k}x+x_0))}K(2^{-k}x+x_0,2^{-k}y) 
$$
satisfies the same assumptions as $K$. From our choice of $\alpha$ it also follows that
$$
\Big|2^{(\alpha k+1)(p(2^{-k}x+x_0)-1)-k(s(2^{-k}x+x_0)p(2^{-k}x+x_0))}\tilde f\Big|\leq \e \text{ in $B_1$}.
$$
Supposing that $|\{v\leq 0\}\cap B_1|\geq |B_1|/2$ and observing that as before
$$
v(y)\leq 2|2y|^\eta-1, \quad \text{ for $|y|>1$},
$$
we see that $v$ satisfies all the assumptions of Lemma \ref{lem:key2}. The choice $\delta = |B_1|/2$ yields
$$
v(x)\leq 1-\theta \text{ in }B_\frac12,$$
which again implies
$$
u(x)\leq b_k+2^{-\alpha(k+1)}.
$$
Thus the choice $b_{k+1}=b_k$ and $a_{k+1}=b_k+2^{-\alpha(k+1)}$ settles \eqref{eq:akbkvar} for the step $j=k+1$. Hence, we arrive at the estimate
$$
 \osc_{B_r(x_0)} u\leq 2^\alpha r^\alpha .
$$
Recalling our rescaling factor in the beginning and rescaling back to our original $u$ yields
\begin{align*}
&\osc_{B_r(x_0)} u\\
&\leq 2^\alpha\left(2\|u\|_{L^\infty(\R^n)}+2^{\frac{p_1-1}{p_0-1}}\max\Big\{\left(\frac{\|f\|_{L^\infty(B_2)}}{\e}\right)^\frac{1}{p_0-1},\left(\frac{\|f\|_{L^\infty(B_2)}}{\e}\right)^\frac{1}{p_1-1}\Big\}\right)^{-1} r^\alpha \\
&\leq  C\left(\|u\|_{L^\infty(\R^n)}+\max\left(\|f\|_{L^\infty(B_2)}^\frac{1}{p_0-1},\|f\|_{L^\infty(B_2)}^\frac{1}{p_1-1}\right)\right) r^\alpha,
\end{align*}
which is the desired result.

\end{proof}

\bibliography{ref.bib} 
\end{document}